\documentclass{amsart}

\usepackage{amssymb,latexsym,amscd,amsmath, amsthm, bold-extra, enumitem, hyperref, dsfont, mathtools}

\usepackage{color}

\numberwithin{equation}{section}

\long\def\eatit#1{}

\newtheorem{Thm}{Theorem}[section]
\newtheorem{Prop}[Thm]{Proposition}
\newtheorem{Lem}[Thm]{Lemma}
\newtheorem{Cor}[Thm]{Corollary}

\theoremstyle{definition}
\newtheorem{Def}[Thm]{Definition}
\newtheorem{Ex}[Thm]{Example}
\newtheorem{Rmk}[Thm]{Remark}

\newcommand{\PP}{{\mathbb{P}}}
\newcommand{\CC}{{\mathbb{C}}}
\newcommand{\RR}{{\mathbb{R}}}
\newcommand{\ZZ}{{\mathbb{Z}}}

\newcommand{\hadamardot}{\star}
\newcommand{\sqfrhadamardot}{\underline{\hadamardot}}
\newcommand{\Lzeroone}{$\textup{[}01\textup{]}$}
\newcommand{\Lzerotwo}{$\textup{[}02\textup{]}$}
\newcommand{\Lzerothree}{$\textup{[}03\textup{]}$}
\newcommand{\Lonetwo}{$\textup{[}12\textup{]}$}
\newcommand{\Lonethree}{$\textup{[}13\textup{]}$}
\newcommand{\Ltwothree}{$\textup{[}23\textup{]}$}
\newcommand{\Mzeroone}{$\{01\}$}
\newcommand{\Mzerotwo}{$\{02\}$}
\newcommand{\Mzerothree}{$\{03\}$}
\newcommand{\Monetwo}{$\{12\}$}
\newcommand{\Monethree}{$\{13\}$}
\newcommand{\Mtwothree}{$\{23\}$}

\begin{document}


\title{Hadamard products of linear spaces}

\author{Cristiano Bocci, Enrico Carlini and Joe Kileel}

\address{Cristiano Bocci\\
Department of Information Engineering and Mathematics, University of Siena\\
Via Roma, 56 Siena, Italy}
\email{cristiano.bocci@unisi.it}

\address{Enrico Carlini\\
School of Mathematical Sciences, Monash University\\
Wellington Road, Clayton, Australia
\vskip0.1cm Department of Mathematical Sciences, Politecnico di Torino\\
Corso Duca degli Abruzzi 24, Turin, Italy}
\email{enrico.carlini@monash.edu, enrico.carlini@polito.it}

\address{Joe Kileel\\
Department of Mathematics, University of California\\
 Berkeley, CA 94720, USA}
\email{jkileel@math.berkeley.edu}


\begin{abstract} We describe properties of Hadamard products of algebraic varieties.  We show any Hadamard power of a line is a linear space, and we construct star configurations from products of collinear points.  Tropical geometry is used to find the degree of Hadamard products of other linear spaces.

\end{abstract}

\keywords{Hadamard products, linear spaces, star configurations, tropical geometry, bracket polynomials}
\subjclass[2010]{14T05, 14N05, 68W30, 14M99}

\maketitle

\section{Introduction}

The concept of Hadamard product, as matrix entry-wise multiplication, is well known in linear algebra: it has nice properties in matrix analysis (\cite{H1,HM,Liu}) and has applications in both statistics and physics (\cite{Liu2,LNP,LNP2,LT}).
Recently, in the papers \cite{CMS, CTY}, the authors use this entry-wise multiplication to define a Hadamard product between projective varieties: given varieties $X,Y\subset\PP^n$,  their Hadamard product $X \hadamardot Y$ is the closure of the image of the rational map
\[
X \times Y \dashrightarrow \PP^n, \quad  ([a_0:\dots : a_n], [b_0: \dots :b_n])\mapsto [a_0b_0 :a_1b_1:\ldots :a_nb_n].
\]
For any projective variety $X$, we may consider its Hadamard square $X^{\hadamardot 2} = X \hadamardot X$ and its higher Hadamard powers $X^{\hadamardot r} = X \hadamardot X^{\hadamardot (r-1)}$.
In \cite{CMS}, the authors use this definition to describe the algebraic variety associated to the restricted Boltzmann machine, which is the undirected graphical model for binary random variables specified by the bipartite graph $K_{r,n}$. This variety is the $r-$th Hadamard power of the first secant variety of $(\PP^1)^n$. Note that  \cite{CTY} concerns the case $r=2, n=4$.

Hadamard products and powers are in fact well-connected to other operations of varieties.  They are the multiplicative analogs of joins and secant varieties, and in tropical geometry, tropicalized Hadamard products equal Minkowski sums.  It is natural to study properties of this new operation, and see its effect on various varieties.  This paper is a first step in that direction.   Here is how it is organized.

In Section 2, we start by giving a different definition of the Hadamard product of varieties in terms of projections of Segre products. As a first important result we give a Hadamard version of Terracini's Lemma (Lemma \ref{Terracini}) which describes the tangent space of $X\hadamardot Y$ at $p\hadamardot q$ as the span $\langle p\hadamardot T_ q (Y),q\hadamardot T_ p (X)  \rangle$.  We also point out varieties parametrized by monomials are closed under Hadamard product.

In Section 3, our attention is fixed on the Hadamard powers of a line.  This case is special and admits direct analysis by projective geometry.  Theorem \ref{mainthmLr} is that the powers of a line are linear spaces, and our proof uses the Hadamard version of Terracini's Lemma.  Proposition \ref{Plucker} then offers explicit equations.

In Section 4, we study the $r$-th square-free Hadamard power $Z^{\sqfrhadamardot r}$ of a finite set $Z$ of projective points.  We obtain a classification when $Z$ is collinear, using star configurations.  In contrast with the standard approach to construct star configurations as intersections of a set of randomly chosen linear spaces, which can give points with complicated coordinates,  Theorem \ref{final} permits a cheaper construction, easily implementable in computer algebra software.

In Section 5, we recall basic definitions in tropical geometry, and give the precise connection to Hadamard products.  Motivated by tropical considerations, we define a refined expected dimension formula for Hadamard products.  Interestingly, Hadamard products can have deficient dimension (Example \ref{dimcounterexample}).

In Section 6, we analyze Hadamard products of linear spaces in general.  This case requires the use of tropical machinery, and  we derive a result analogous to identifiability in the theory of secant varieties along the way.  The paper culminates with a general degree formula for Hadamard products of linear spaces (Theorem \ref{degformula}),  an extension to include \textit{reciprocal linear spaces} (Corollary \ref{reciprocal}), and then two non-trivial bracket polynomial formulas.

\vspace{.2cm}

\section{Preliminaries}
We work over the field of complex numbers $\mathbb{C}$.

\begin{Def}
Let $H_i\subset\PP^n,i=0,\ldots,n$, be the hyperplane $x_i=0$ and set
$$\Delta_i=\bigcup_{0\leq j_1<\ldots<j_{n-i}\leq n}H_{j_1}\cap\ldots\cap H_{j_{n-i}}.$$
\end{Def}

In other words, $\Delta_i$ is the $i-$dimensional variety of points having {\em at most} $i+1$ non-zero coordinates. Thus $\Delta_0$ is the set of coordinates points and $\Delta_{n-1}$ is the union of the coordinate hyperplanes. Note that elements of $\Delta_i$ have {\em at least} $n-i$ zero coordinates. We have the following chain of inclusions:
\begin{equation}\label{inclusiondelta}
\Delta_0=\{[1:0:\ldots:0],\ldots,[0:\ldots:0:1]\}\subset\Delta_1\subset\ldots\subset\Delta_{n-1}\subset\Delta_{n}=\PP^n.\end{equation}

\begin{Def}\label{projectiondef} Given varieties $X,Y\subset\PP^n$ we consider the usual Segre product
\[X\times Y\subset\PP^N\]
\[([a_0:\dots : a_n], [b_0: \dots :b_n])\mapsto [a_0b_0 :a_0b_1:\ldots :a_nb_n]\]
and we denote with $z_{ij}$ the coordinates in $\PP^N$. Let $\pi:\PP^N \dashrightarrow \PP^n$ be the projection map from the linear space $\Lambda$ defined by equations $z_{ii}=0,i=0,\ldots,n$. The {\em Hadamard product} of $X$ and $Y$ is
\[X\hadamardot Y=\overline{\pi(X\times Y)},\]
where the closure is taken in the Zariski topology.
\end{Def}


\begin{Def} Let $p,q\in \PP^n$ be two points of coordinates respectively $[a_0:a_1:\ldots:a_n]$ and $[b_0:b_1:\ldots:b_n]$. If $a_ib_i\not= 0$ for some $i$, their Hadamard product $p\hadamardot q$ of $p$ and $q$, is defined as
\[
p\hadamardot q=[a_0b_0:a_1b_1:\ldots:a_nb_n].
\]
If  $a_ib_i= 0$ for all $i=0,\dots, n$ then we say $p\hadamardot q$ is not defined.
\end{Def}

\begin{Rmk}\label{Hadamrdasunion} We have that
\[X\hadamardot Y=\overline{\{p\hadamardot q: p\in X, q\in Y, p\hadamardot q\mbox{ is defined} \}}.\]
Hence, our definition of Hadamard product corresponds with the one in \cite{CMS} and \cite{CTY}.
\end{Rmk}

\begin{Rmk}\label{limitremark} To understand the role of the closure operation in Definition \ref{projectiondef} of Hadamard product we can proceed as follows. If $\pi$ is defined on all of $X \times Y$, then no closure is needed, by \cite[Corollary 14.2]{Eis}.   Otherwise, we can blow up along $\Lambda\cap (X\times Y)$. The points of $X\hadamardot Y$ are hence limits of $p(t)\hadamardot q(s)$ for curves $p(t)\in X$ and $q(s)\in Y$, in the Euclidean topology by \cite[\textsection I.10, Corollary 1]{M}.  Thus the closure operation only adds points obtained when $p(t)\hadamardot q(s)$ goes to a point of $\Lambda\cap (X\times Y)$.
\end{Rmk}

\begin{Rmk}
In general, it is difficult to compute $X \hadamardot Y$.  The situation seems as bad as computing joins.  Suppose we have for input the defining homogeneous ideals $I$ and $J$ of $X$ and $Y$ in $\PP^n$, and we wish to output the defining ideal $K$ of $X \hadamardot Y$.  It may be achieved with an elimination as follows:
\begin{itemize}
\item Work in the ring $\CC[x_0, \ldots, x_n, y_0, \ldots, y_n, z_0, \ldots, z_n]$ with $3n+3$ variables.
\item Form the ideal $I(x) + J(y) + \langle z_0 - x_0y_0, \ldots, z_n - x_ny_n \rangle$.
\item Eliminate the $2n+2$ variables $x_0, \ldots, y_n$.
\end{itemize}
However, this elimination is a heavy computation.  We know of no efficient algorithm to compute Hadamard products in general.
\end{Rmk}

Note that $X\hadamardot Y$ is a variety such that $\dim (X\hadamardot Y)\leq  \dim (X)+\dim (Y)$ and that $X\hadamardot Y$ can be empty even if neither $X$ nor $Y$ is empty.  In Section 5, we will give a refined upper bound on $\dim(X\hadamardot Y)$.

Note that if $X, Y$ are irreducible, then $X \times Y$ is irreducible, so $X \hadamardot Y =  \overline{\pi(X\times Y)}$ is irreducible.  Also, associativity of Hadamard product of points extends to associativity of Hadamard product of varieties:

\begin{Lem}\label{assoc}
If $X,Y,Z\subset\PP^n$ are varieties, then $(X\hadamardot Y)\hadamardot Z = X\hadamardot (Y\hadamardot Z)$.
\end{Lem}

\begin{proof} In fact, both sides equal
\[ \Pi = \overline{\{p\hadamardot q \hadamardot z: p\in X, q\in Y, z\in Z, p\hadamardot q\hadamardot z \mbox{ is defined} \}}. \]
\noindent To see $(X\hadamardot Y)\hadamardot Z \subset \Pi$, let $s \in X\hadamardot Y, z\in Z$ and $f$ be a homogeneous polynomial in the ideal $I(\Pi)$.  Consider $f(z_0 \cdotp \-- , z_1 \cdotp \-- , \ldots , z_n \cdotp \--)$ a polynomial defined up to scale.  It vanishes on ${\{p\hadamardot q : p\in X, q\in Y, p\hadamardot q \mbox{ is defined}\}}$, and hence at $s$ in the closure.  The other inclusions then are clear.
\end{proof}

In what follows we will often work with Hadamard products of a variety with itself, thus we give the following:

\begin{Def} Given a positive integer $r$ and a variety $X\subset\PP^n$, the {\em r-{th} Hadamard power} of $X$  is
\[X^{\hadamardot r}= X\hadamardot X^{\hadamardot (r-1)},\]
where $X^{\hadamardot 0}=[1: \dots : 1]$.
\end{Def}
Note that $\dim (X^{\hadamardot r})\leq r\dim (X) $ and the $r$-th Hadamard power cannot be empty if $X$ is not empty.

As the definition of Hadamard product involves a closure operation, it is not trivial to describe all points of $X^{\hadamardot r}$. However, we can say something about the maximal number of zero coordinates of any point in the Hadamard power.

\begin{Lem}\label{numberofzeroes}
If $X\cap\Delta_{n-i}=\emptyset$, then $X^{\hadamardot r}\cap\Delta_{n-ri+r-1}=\emptyset$.

\end{Lem}
\begin{proof}

The hypothesis means that all points of $X$ have at most $i-1$ zero coordinates. By Remark \ref{limitremark} any point in $X^{\hadamardot r}$ is the limit of the Hadamard product of $r$ points in $X$. Thus, any point is $X^{\hadamardot r}$ has at most $ri-r$ zero coordinates. Hence, $X^{\hadamardot r}\cap\Delta_{n-ri+r-1}=\emptyset$.
\end{proof}

In the case of lines we have a stronger result.
\begin{Lem}\label{closurelem}
Let $n\geq 2$ and $n\geq r$. If $L\subset\PP^n$ is a line such that $L\cap\Delta_{n-2}=\emptyset$, then
\[L^{\hadamardot r}=\bigcup_{p_i\in L} p_1\hadamardot\ldots\hadamardot p_r,\]
that is, the closure operation is not necessary.
\end{Lem}
\begin{proof} We use Remark \ref{limitremark} and we note that the map $\pi$ is everywhere defined.
\end{proof}

\begin{Ex} For contrast, consider the curve $X\subset \PP^2$ defined by $x_0x_1-x_2^2=0$ and the point $p=[1:0:2]$.
The Hadamard product $p\hadamardot X$ is the line $x_1=0$. However the point $[0:0:1]$ cannot be obtained as $p\hadamardot x$, with $x\in X$.
\end{Ex}

We now prove the analogue of Terracini's Lemma for Hadamard products using tangent spaces; recall that for a variety $X$ and a point $p\in X$, $T_p(X)$ denotes the {\em tangent space} to $X$ at $p$.
\begin{Lem}\label{Terracini} Consider varieties $X,Y\subset\PP^n$. If  $p\in X$ and $q\in Y$  are general points, then
\[T_{p\hadamardot q}(X\hadamardot Y)=\langle p\hadamardot T_ q (Y),q\hadamardot T_ p (X)  \rangle.\]
Moreover, if $p_1,\ldots,p_r\in X$ are general points and $p_1\hadamardot\ldots\hadamardot p_r\in X^{\hadamardot r}$ is a general point, then
\[T_{p_1\hadamardot\ldots\hadamardot p_r}(X^{\hadamardot r})=\langle p_2\hadamardot\ldots\hadamardot p_r\hadamardot T_ {p_1} (X),\ldots,p_1\hadamardot\ldots\hadamardot p_{r-1}\hadamardot T_ {p_{r}} (X)  \rangle.\]
\end{Lem}
\begin{proof}
It is enough to prove the result for $X$ and $Y$. Since $p$ and $q$ are generic, by generic smoothness, we can find a parametrization of $X$, respectively of $Y$, in an analytic neighbourhood of $p$, respectively of $q$. Let $p(x)$ with $p(0)=p$, respectively $q(y)$ with $q(0)=q$, be such a parametrization. Thus, a parameterization of $X\hadamardot Y$ around $p\hadamardot q$, using our base field is $\CC$ and Proposition 10.6 in \cite{Hart}, is given by
\[p(x)\hadamardot q(y).\]
Thus the tangent space $T_{p\hadamardot q}(X\hadamardot Y)$ is obtained by picking curves $x(t)$ and $y(t)$, taking the derivative with respect to $t$ at $t=0$ in the expression
\[p(x(t))\hadamardot q(y(t))\]
and making $x'(0)$ and $y'(0)$ vary.
The result follows.
\end{proof}

We note in passing that toric varieties, in the sense of \cite{S}, are closed under Hadamard products.

\begin{Rmk} Let $\mathcal{A} = (\textbf{a}_0, \ldots, \textbf{a}_n)$ be a $d \times (n+1)$ integer matrix in which all column sums are the same.  This defines $X \subset \PP^n$ as the closure of the image of the monomial map:

\begin{center}
$ (\CC^{\times})^{d} \longrightarrow \PP^n $

\hspace {2.4cm} $ \textbf{t} \longmapsto [\textbf{t}^{\textbf{a}_0} : \ldots : \textbf{t}^{\textbf{a}_n}]$.
\end{center}

\noindent Let $\mathcal{B}$ be an $e \times (n+1)$ integer matrix similarly defining $Y \subset \PP^n$.  Then $X \hadamardot Y$ is defined by the $(d+e) \times (n+1)$ integer matrix which is the vertical concatenation of $\mathcal{A}$ and $\mathcal{B}$.
\end{Rmk}

It would be interesting to study Hadamard products of \textit{T-varieties with low complexity} (\cite {IS}) in future work.

\vspace{.2cm}

\section{Powers of a line}

In this section, our aim is to give a direct and elementary proof by projective geometry that Hadamard powers of a line are linear spaces.  This is achieved in Theorem \ref{mainthmLr}, and made explicit in Proposition \ref{Plucker}.

\subsection{Multiplying by a point}
We begin with useful results about multiplying linear spaces by a point, which amounts to projecting into a coordinate subspace, then acting by an element of that subspace's torus.
\begin{Lem}\label{pLLem}
Let $L\subset\PP^n$ be a linear space of dimension $m$. If $p\in\PP^n$, then $p\hadamardot L$ is either empty or it is a linear space of dimension at most $m$. If $p\not\in\Delta_{n-1}$, then $\dim( p\hadamardot L)=m$.
\end{Lem}
\begin{proof}
We first consider the case $p\not\in\Delta_{n-1}$, that is, $p$ has no coordinate equal to zero. We can describe $L$ as the solution set of a linear system of equations of the form $M\mathbf{x}=0$.
Let $D_p$ be the diagonal matrix with entries a choice of coordinates of $p$ and consider the matrix $M'=M{D_p}^{-1}$. If $q\in L$, then $p\hadamardot q$ solves $M'\mathbf{x}=0$. Conversely, if $z$ solves $M'\mathbf{x}=0$, then $z=p\hadamardot q$ for some $q\in L$. Thus, we have
\[\bigcup_{q\in L}p\hadamardot q=\{\mathbf{x}:M'\mathbf{x}=0\}.\]
Since linear spaces are closed in the Zariski topology, it follows from Remark \ref{Hadamrdasunion} that $p\hadamardot L=\{\mathbf{x}:M'\mathbf{x}=0\}$. This completes the proof since $M$ and $M'$ have the same rank.

Now we consider the case $p\in\Delta_{i}\setminus\Delta_{i-1}$ for $i<n$, that is, $p$ has exactly $n-i$ coordinates equal to zero. This means that we can find $i+1$ coordinates points spanning a linear space $\Sigma$ such that $p\in\Sigma$; let $\Lambda$ be the linear span of the remaining $n-i$ coordinate points. Now consider the projection map from $\Lambda$
$$\pi:\PP^n\dashrightarrow\Sigma.$$
Choose $q\in\pi^{-1}(p)\setminus\Delta_{n-1}$.
The conclusion follows since $p\hadamardot L=\pi(q\hadamardot L)$.
\end{proof}

\begin{Lem}\label{pqLLem}
Let $L\subset\PP^n$ be a linear space of dimension $m<n$ and consider points $p,q\in\PP^n\setminus\Delta_{n-1}$. If $p\neq q$, $L\cap\Delta_{n-m-1}=\emptyset$, and $\langle p,q\rangle\cap\Delta_{n-m-2}=\emptyset$, then $p\hadamardot L\neq q\hadamardot L$.
\end{Lem}
\begin{proof}
Let $p=[a_0: \dots: a_n]$ and $q=[b_0: \dots: b_n]$. Let $t=n-m$. We prove the result by contradiction dealing with the case $t=1$ and then using a projection argument.

\noindent{\bf Case $t=1$.} In this situation, $L$ is a hyperplane.  We consider its equation
\[L: \sum_{i=0}^{n} \alpha_i x_i=0.\]
Direct computations show that
\[p\hadamardot L:\sum_{i=0}^{n} \frac{\alpha_i}{a_i} x_i=0 \: \: \mbox{ and } \: \:q\hadamardot L:\sum_{i=0}^{n} \frac{\alpha_i}{b_i} x_i=0.\]
Since $L\cap\Delta_{0}=\emptyset$ all the coefficients $\alpha_i$ are non-zero. The assumption $p\hadamardot L= q\hadamardot L$ yield that $\frac{a_i}{b_i}=\frac{a_j}{b_j}$ for all $i,j$. This is a contradiction as $p\neq q$.

\noindent{\bf Case $t>1$.} Choose $t-1$ coordinate points in $\Delta_0$ and denote with $\Sigma$ their linear span. Note that $\Sigma\subset\Delta_{t-2}$. Let $\pi$ be the projection from $\Sigma$
\[\pi:\PP^n\dashrightarrow\PP^{n-t+1},\]
that is, $\pi$ forgets $t-1$ coordinates. We set
 \[p'=\pi(p),q'=\pi(q)\mbox{ and }L'=\pi(L).\]
Since $L\cap\Sigma\subset L\cap\Delta_{t-1}\subset L\cap\Delta_{n-m-1}=\emptyset$, $L'$ is a linear space of dimension $m$. Similarly, $p'\neq q'$ since $\langle p,q\rangle\cap\Delta_{t-2}=\emptyset.$
Since $\pi$ simply forgets $t-1$ coordinates, we have
\[p'\hadamardot L'=\pi(p\hadamardot L)\mbox{ and }q'\hadamardot L'=\pi(q\hadamardot L).\]

\noindent Now, we show $p'\hadamardot L'\neq q'\hadamardot L'$, and hence $p\hadamardot L\neq q\hadamardot L$.
To conclude using the case $t=1$, we have to check the following in $\mathbb{P}^{m+1}$:
\begin{enumerate}
\item\label{firstonc} $L'\cap\Delta_{0}=\emptyset$.
\item\label{secondcond}$p',q'\not\in\Delta_{m}=\emptyset$.
\end{enumerate}
Since $p\notin  \Delta_{n-1}$, clearly $p'\notin \Delta_m$, and similarly $q'\notin \Delta_m$. Thus \eqref{secondcond} is proved. If $z'\in L'\cap\Delta_0$, then we can choose $z\in L$ such that $\pi(z)=z'$. Thus, $z\in\Delta_{t-1}=\Delta_{n-m-1}$, a contradiction. Hence \eqref{firstonc} is proved and the proof is now completed.
\end{proof}

\vspace{.1cm}

\subsection{Powers of a line}

The case of the powers of a line is special.  The main reason is that two dimension counts agree in the proof of \ref{mainthmLr}.

\begin{Lem}\label{linenoclosurelem}
Let $n\geq 2, \, n\geq r$ and $L\subset\mathbb{P}^n$ be a line such that $L\cap\Delta_{n-2}=\emptyset$.  There exist $p\neq q$ in $L\setminus\Delta_{n-1}$, and for any such: $z\in L$ implies $z\hadamardot L^{\hadamardot r} \subset \langle p\hadamardot L^{\hadamardot r}, \, q\hadamardot L^{\hadamardot r}\rangle$.

\end{Lem}
\begin{proof}
Choose a vector space basis of $L$ made of points with no zero coordinates. One exists, since $L\cap\Delta_{n-2}=\emptyset$. Hence $z$ is a linear combination of $p$ and $q$ and:
\[
z \hadamardot \bigcup_{z_i\in L} z_1\hadamardot\ldots\hadamardot z_r \, \, \subset \, \, \langle p\hadamardot L^{\hadamardot r}, \, q\hadamardot L^{\hadamardot r}\rangle.
\]
We can note the LHS equals $L^{\hadamardot r}$ by Lemma \ref{closurelem}, or take closure of both sides.
\end{proof}

\begin{Thm}\label{mainthmLr}Let $L\subset\PP^n$, $n>1$, be a line. If $L\cap\Delta_{n-2}=\emptyset$,
then $L^{\hadamardot r}\subset\PP^n$ is a linear space of dimension $\min\{r,n\}$.
\end{Thm}
\begin{proof}
It is enough to consider the case $r\leq n$. Indeed, if $r>n$, $\PP^n\supseteq L^{\hadamardot r}\supseteq p\hadamardot L^{\hadamardot n}$ for some $p\in L^{\hadamardot (r-n)}$ without zero coordinates. The result follows from Lemma \ref{pLLem}.

We will prove the statement by proving that $(i)$ $\dim( L^{\hadamardot r})=r$ and that $(ii)$ $\dim \langle L^{\hadamardot r}\rangle =r$. We proceed by induction on $r$.  The base case is $r=1$.


Now we assume the statement to hold for $r$ and we prove it for $r+1$. To prove $(i)$ we use Lemma \ref{Terracini}. Thus, we consider general points $p_1,\ldots,p_{r+1}\in L\setminus\Delta_{n-1}$ and we have
\[T_{p_1\hadamardot\ldots\hadamardot p_{r+1}}(L^{\hadamardot (r+1)})=\langle p_2\hadamardot\ldots\hadamardot p_{r+1}\hadamardot L,\ldots,p_1\hadamardot\ldots\hadamardot p_{r}\hadamardot L  \rangle\]
\[\qquad \qquad \qquad \qquad \qquad \: \: \: \: \: =\langle p_{r+1}\hadamardot T_{p_1\hadamardot\ldots\hadamardot p_{r}}(L^{\hadamardot r}),p_{1}\hadamardot T_{p_2\hadamardot\ldots\hadamardot p_{r+1}}(L^{\hadamardot r}) \rangle.\]
Using the inductive hypothesis we conclude that
\[T_{p_1\hadamardot\ldots\hadamardot p_{r+1}}(L^{\hadamardot (r+1)})=\langle p_{r+1}\hadamardot L^{\hadamardot r},p_{1}\hadamardot L^{\hadamardot r} \rangle.\]
Thus $(i)$ follows, since $p_{r+1}\hadamardot L^{\hadamardot r}$ and $p_{1}\hadamardot L^{\hadamardot r}$ are $r$ dimensional, distinct (by Lemma \ref{pqLLem} which we can apply because of Lemma \ref{numberofzeroes}) linear spaces intersecting in an $r-1$ dimensional linear space, namely $p_1\hadamardot p_{r+1}\hadamardot L^{\hadamardot r-1}$, so we conclude
\[\dim( L^{\hadamardot (r+1)})=r+1.\]
To prove $(ii)$ we use Lemma \ref{linenoclosurelem} which yields
\[\langle L^{\hadamardot (r+1)}\rangle=\langle \bigcup_{z\in L} z\hadamardot L^{\hadamardot r}\rangle\subset\langle p\hadamardot L^{\hadamardot r},q\hadamardot L^{\hadamardot r}\rangle\]
for $p\neq q$ and $p,q\in L\setminus\Delta_{n-1}$.
Again by Lemma \ref{numberofzeroes}, Lemma \ref{pLLem}, Lemma \ref{pqLLem}, and the inductive hypothesis we get that $p\hadamardot L^{\hadamardot r},q\hadamardot L^{\hadamardot r}$ are distinct linear spaces of dimension $r$ containing a common linear space of dimension $r-1$, namely $p\hadamardot q\hadamardot L^{\hadamardot (r-1)} $. Hence, $\dim (\langle L^{\hadamardot (r+1)}\rangle)=r+1$ and the proof is now completed.
\end{proof}

\begin{Rmk}
If $L\cap\Delta_{n-2}=\emptyset$ fails in Theorem \ref{mainthmLr}, $L^{\hadamardot r}$ is still linear, but possibly of lower dimension.  For example, consider the line $L\subset \PP^5$
\[
L: \begin{cases}
2x_0-x_1=0\\
x_1+3x_2-x_4=0\\
3x_2-x_3=0\\
16x_3-12x_4-3x_5=0.
\end{cases}
\]
In \texttt{Macaulay2} (\cite{M2}), we compute
\[
L^{\hadamardot 2}: \begin{cases}
9x_2-x_3=0\\
192x_1+64x_3-48x_4-9x_5=0\\
768x_0+64x_3-48x_4-9x_5=0.
\end{cases}
\]
\vspace{.1cm}
\[
L^{\hadamardot 3}: \begin{cases}
27x_2-x_3=0\\
8x_0-x_1=0
\end{cases}
\]
\vspace{.1cm}
\[
L^{\hadamardot 4}: \begin{cases}
81x_2-x_3=0\\
16x_0-x_1=0
\end{cases}
\]

\noindent and $\dim(L^{\hadamardot r})=3$ for all $r\geq 3$.
\end{Rmk}

\vspace{.1cm}

\subsection{Equations}

We now make Theorem \ref{mainthmLr} explicit, and express the Pl\"ucker coordinates (\cite[\textsection III.14]{MillS}) of $L^{\hadamardot r}$ in terms of the Pl\"ucker coordinates of the given line $L$ in $\mathbb{P}^n$.  From this, we get canonical equations for $L^{\hadamardot r}$.

\begin{Prop}\label{Plucker}
Let $L$ in $\mathbb{P}^{n}$, $n>1$, be a line with $L \cap \Delta_{n-2} = \emptyset$, and let $r<n$.

\begin{enumerate}[label={$(\arabic*)$}]

\item If $L = \textup{rowspan}\begin{pmatrix}
a_{00} & a_{01} & \dots & a_{0n}\\
a_{10} & a_{11} & \dots & a_{1n}
\end{pmatrix}$, then

$L^{\hadamardot r} = \textup{rowspan}\begin{pmatrix}
a_{00}^{r} & a_{01}^{r}  & \dots & a_{0n}^{r}\\
a_{00}^{r-1}a_{10} & a_{01}^{r-1}a_{11}  & \dots & a_{0n}^{r-1}a_{1n} \\
\dots & \dots & \dots & \dots \\
a_{10}^{r} & a_{11}^{r} & \dots & a_{1n}^{r}
\end{pmatrix}$.\\

\item In terms of Pl\"ucker coordinates, if $0 \le i_{0} < i_{1} < \dots < i_{r} \le n$, then

\textup{[}$i_{0}, i_{1}, \dots, i_{r}\textup{]}_{L^{\hadamardot r}} =  \displaystyle\prod_{0 \le j<k\le r} \textup{[}i_{j}, i_{k}\textup{]}_{L}$ .

\end{enumerate}

\end{Prop}

\begin{proof}

By Lemma \ref{linenoclosurelem}, if $L$ is spanned by points $p$ and $q$, then $L^{\hadamardot r}$ is spanned by points $p^{\hadamardot r} \, , \, p^{\hadamardot (r-1)} \hadamardot q \, , \, \dots \, , \, q^{\hadamardot r}.$  Part (1) follows.  For (2): \\

\begin{center}

$\textup{[}i_{0}, i_{1}, \dots, i_{r}\textup{]}_{L^{\hadamardot r}} = \textup{det}\begin{pmatrix}
a_{0i_{0}}^{r} & a_{0i_{1}}^{r}  & \dots & a_{0i_{r}}^{r}\\
a_{0i_{0}}^{r-1}a_{1i_{0}} & a_{0i_{1}}^{r-1}a_{1i_{1}}  & \dots & a_{0i_{r}}^{r-1}a_{1i_{r}} \\
\dots & \dots & \dots & \dots \\
a_{1i_{0}}^{r} & a_{1i_{1}}^{r} & \dots & a_{1i_{r}}^{r}
\end{pmatrix}$. \end{center}

\noindent This is a homogenized Vandermonde determinant.  It equals:

\vspace{.2cm}

 \hspace{2.1cm}
$\displaystyle\prod_{0 \le j<k\le r} (a_{0i_{j}}a_{1i_{k}} - a_{0i_{k}}a_{1i_{j}}) \: \: \: = \displaystyle\prod_{0 \le j<k\le r} \textup{[}i_{j}, i_{k}\textup{]}_{L}.$
\end{proof}

The virtue of expressing a formula in terms of Pl\"ucker coordinates as above, and also as later on in this paper, is that Pl\"ucker coordinates form the most canonical coordinate system on a Grassmannian of linear spaces.

\begin{Cor}\label{lineareqn} In Proposition \ref{Plucker}, if $x_{0} \, , \, x_{1} \, , \, \dots \, , \, x_{n}$ are the homogeneous coordinates of $\mathbb{P}^{n}$, then $L^{\hadamardot r}$ is cut out by the maximal minors of

\begin{center}

$\begin{pmatrix}
a_{00}^{r} & a_{01}^{r}  & \dots & a_{0n}^{r}\\
a_{00}^{r-1}a_{10} & a_{01}^{r-1}a_{11}  & \dots & a_{0n}^{r-1}a_{1n} \\
\dots & \dots & \dots & \dots \\
a_{10}^{r} & a_{11}^{r} & \dots & a_{1n}^{r}\\
x_{0} & x_{1} & \dots & x_{n}
\end{pmatrix}$.

\end{center}

\noindent These are ${n+1 \choose r+2}$ linear equations with coefficients the Pl\"ucker coordinates of $L^{\hadamardot r}$ (up to sign).  For $r < n-1$, they are not a complete intersection.  But for $r = n-1$, there is just one equation, for the hyperplane $L^{\hadamardot (n-1)}:$

\begin{center}
$L^{\hadamardot (n-1)} \: \: : \: \: \displaystyle\sum_{i=0}^n \Big(  (-1)^{n+i} \displaystyle\prod_{\substack{0 \le j<k\le n \\ j, k \, \neq \, i}} \textup{[}j, k\textup{]}_{L} \: \Big) \, x_{i} \: \: = \: \: 0$.
\end{center}
\end{Cor}

\vspace{.2cm}

\section{Star Configurations}

In this section, we give an application of the Hadamard powers of a line: Theorem \ref{final}, which can be viewed as a simple construction of star configurations (\cite {CGVT, CVT, GHM}). Interest in star configurations has increased in recent years because of their extremal behaviour with respect to many open problems such as, for example, symbolic powers of ideals \cite{BH10}.

Given a finite set of points $Z\subset\PP^n$, we can consider the Hadamard powers $Z^{\hadamardot r}$. However, it is more interesting to only consider Hadamard products of distinct points in $Z$.

\begin{Def}
Let $Z\subset\PP^n$ be a finite set of points. The $r-$th {\em square-free} Hadamard power of $Z$ is
\[Z^{\sqfrhadamardot r}=\{p_1\hadamardot\ldots\hadamardot p_r : p_i\in Z\mbox{ and }p_i\neq p_j \mbox{ for }i\neq j\}.\]
\end{Def}

We want to describe the square-free Hadamard powers of a finite set of points. Thus, we begin with a useful technical result.

We will use the following notion.
{\begin{Def} Let $H_1,\ldots ,H_m\subset M\subset \PP^n$ be linear spaces such that $r=\dim M=\dim H_i+1$. The linear spaces $H_i$  are said to be in {\em linear general position in $M$} if, whenever $i_1, \dots , i_j$ are distinct:
\begin{itemize}
\item $\dim (H_{i_1}\cap\ldots\cap H_{i_j})=r-j$,\, for $j\leq r$.
\item $H_{i_1}\cap\ldots\cap H_{i_j}=\emptyset$,\, for $r<j$.
\end{itemize}
\end{Def}
\begin{Lem}\label{lingenpos}  Let $L\subset\PP^n$ be a line, let $p_i\in L\setminus\Delta_{n-1}, \,1\leq i\leq m$, be $m$ distinct points, and $r\leq n$.  Set $M=L^{\hadamardot r}$ and $H_i=p_i\hadamardot L^{\hadamardot(r-1)},1\leq i\leq m$. If $L\cap\Delta_{n-2}=\emptyset$, then whenever $i_1, \dots ,i_j$ are distinct:
\begin{itemize}
\item $H_{i_1}\cap\ldots\cap H_{i_j}=p_{i_1}\hadamardot\ldots\hadamardot p_{i_j}\hadamardot L^{\hadamardot(r-j)}, \mbox{ for }j\leq r$.
\item $H_{i_1}\cap\ldots\cap H_{i_j}=\emptyset, \mbox{ for }r< j$.
\end{itemize}
In particular, the linear spaces $H_i$ are in linear general position in $M$.
\end{Lem}
\begin{proof} We will prove this for $j\leq r+1$ by induction, and that suffices.

The base case is $j=1$.  Now we assume the statement for $j\leq r$ and prove it for $j+1$.  Without loss of generality it is enough to show that:
\[H_1\cap\ldots\cap H_{j+1}=p_1\hadamardot\ldots\hadamardot p_{j+1}\hadamardot L^{\hadamardot(r-j-1)}.\]

By the inductive hypothesis:

\begin{itemize}
\item $H_1 \cap \ldots \cap H_j = p_1 \hadamardot \ldots \hadamardot p_j \hadamardot L^{\hadamardot (r-j)}$, and
\item $H_2 \cap \ldots \cap H_{j+1} = p_2 \hadamardot \ldots \hadamardot p_{j+1} \hadamardot L^{\hadamardot (r-j)}$.
\end{itemize}
Intersecting these, we get:
\[H_1 \cap \ldots \cap H_{j+1} = p_1 \hadamardot V \, \cap \, p_{j+1} \hadamardot V,\]
where $V = p_2 \hadamardot \ldots \hadamardot p_j  \hadamardot L^{\hadamardot (r-j)}$.
By Theorem \ref{mainthmLr} and Lemma \ref{pLLem}, $V$ is a linear space of dimension $r-j$.  By Lemma \ref{numberofzeroes}, $V \cap \Delta_{n-(r-j)-1} = \emptyset$, which means Lemma \ref{pqLLem} applies to give $p_1 \hadamardot V \neq p_j \hadamardot V$.  Combining with

\[p_1 \hadamardot V \, \cap \, p_{j+1} \hadamardot V \supset p_1\hadamardot\ldots\hadamardot p_{j+1}\hadamardot L^{\hadamardot(r-j-1)},\]

\noindent and Lemma \ref{pLLem} and Theorem \ref{mainthmLr}, we complete the induction.
\end{proof}

\begin{Cor}\label{distinct}
{Let $L\subset\PP^n$ be a line, $r\leq n$, and consider points $p_i,q_j \in L\setminus\Delta_{n-1}$ for $1\leq i,j\leq r$. If $L\cap\Delta_{n-2}=\emptyset$, and
\[\{p_1,\ldots,p_r\}\neq\{q_1,\ldots,q_r\},\]
then $p_1\hadamardot\ldots\hadamardot p_r\neq q_1\hadamardot\ldots\hadamardot q_r$.}
\end{Cor}
\begin{proof}
Set $x=p_1\hadamardot\ldots\hadamardot p_r$ and $y=q_1\hadamardot\ldots\hadamardot q_r$.
Without loss of generality we may assume $q_1\neq p_i$ for $1\leq i\leq r$. Consider the linear spaces
\[p_{1}\hadamardot L^{\hadamardot(r-1)},\, \ldots \, , p_{r}\hadamardot L^{\hadamardot(r-1)}\]
and \[q_{1}\hadamardot L^{\hadamardot(r-1)}.\] They are in linear general position in $L^{\hadamardot r}$ by Lemma \ref{lingenpos}. So:
\begin{itemize}
\item $x=p_{1}\hadamardot L^{\hadamardot(r-1)}\cap\ldots\cap p_{r}\hadamardot L^{\hadamardot(r-1)}$, and
\item $\emptyset = q_{1}\hadamardot L^{\hadamardot(r-1)}\cap p_{1}\hadamardot L^{\hadamardot(r-1)}\cap\ldots\cap p_{r}\hadamardot L^{\hadamardot(r-1)}.$
\end{itemize}
From $y\in q_{1}\hadamardot L^{\hadamardot(r-1)}$, the result follows.
\end{proof}

\begin{Cor}\label{cardinality}
Let $L\subset\PP^n$ be a line, $Z\subset L$ be a set of $m$ points and $r\leq n$. If $L\cap\Delta_{n-2}=\emptyset$ and $Z\cap\Delta_{n-1}=\emptyset$, then $Z^{\sqfrhadamardot r}$ is a set of $\binom{m}{r}$ points.
\end{Cor}

We now see that  $Z^{\sqfrhadamardot r}$ for collinear points is a star configuration.

\begin{Def} A set of $\binom{m}{r}$ points $\mathbb{X}\subset\PP^n$ is a {\em star configuration} if there exist linear spaces $H_1, \dots , H_m \subset M\subset\mathbb{P}^n$ such that:
\begin{itemize}
\item $r=\dim M = \dim H_i + 1$.
\item $H_i$ are in linear general position in $M$.
\item $\mathbb{X}=\bigcup_{1\leq i_1<\ldots< i_r\leq m} H_{i_1}\cap\ldots\cap H_{i_r}$.
\end{itemize}
\end{Def}

\begin{Thm}\label{final} Let $L\subset\PP^n$ be a line, $Z\subset L$ be a set of $m$ points and $r\leq \min\{m,n\}$. If $L\cap\Delta_{n-2}=\emptyset$ and $Z\cap\Delta_{n-1}=\emptyset$, then $Z^{\sqfrhadamardot r}$ is a star configuration in $M=L^{\hadamardot r}$.
\end{Thm}
\begin{proof}

Let $Z=\{p_1,\ldots,p_m\}$ and notice that $Z^{\sqfrhadamardot r}$ is a set of $\binom{m}{r}$ points by Corollary \ref{cardinality}. Using Lemma \ref{lingenpos} we see that the linear spaces
\[H_1=p_1\hadamardot L^{\hadamardot (r-1)},\ldots,H_m=p_m\hadamardot L^{\hadamardot (r-1)}\]
are in linear general position in $M = L^{\hadamardot r}$, and
\[H_{i_1}\cap\ldots\cap H_{i_r}= p_{i_1}\hadamardot\ldots\hadamardot p_{i_r}\]
 for all $1\leq i_1<\ldots<i_r\leq m$. The conclusion follows from the definitions of star configuration and square-free Hadamard product.
\end{proof}

\vspace{.2cm}

\section{Tropical connection}

To study Hadamard products of other linear spaces in Section 6, we will need some machinery from tropical geometry.  We now review the basics.

Given an irreducible variety $X \subset \PP^n$ not contained in $\Delta_{n-1}$, let $I \subset \CC [x_{0}^{\pm}, \ldots , x_{n}^{\pm}]$ be the defining ideal of $X - \Delta_{n-1} \subset (\CC^{*})^{n+1}/\CC^{*}$.  The set
\[
 \text{trop}(X)  = \overline {\{ \textbf{w} \in \RR^{n+1} / \RR \textbf{1}  :  \text{in}_{\textbf{w}}(I) \neq \langle 1 \rangle \}}
\]
is the \textit{tropicalization} of $X$ (\cite[Theorem 3.2.3]{MS}).  It is the support of a pure rational polyhedral subfan of the Gr\"obner fan (\cite[\textsection 8.4]{CLO}) of $I$, inside $\RR^{n+1} / \RR \textbf{1}$ (\cite[Theorem 3.3.5]{MS}).  That subfan has positive integer multiplicities attached to its facets:
\[
\text{mult}(\sigma) \, \, = \sum_{P \in \text{Ass}^{\text{min}}(\text{in}_{\textbf{w}}(I))} \text{mult}\big{(}P, \, \text{in}_{\textbf{w}}(I)\big{)}.
\]
Here $\textbf{w}$ is any point in the relative interior of $\sigma$ and the sum is over minimial associated primes $P$ of $\text{in}_{\textbf{w}}(I)$.  These multiplicites balance along ridges (\cite[Theorem 3.4.14]{MS}).  A lot of the geometry of $X$ can be recovered from trop($X$).

Tropical geometry provides powerful tools to study Hadamard products, because of the following connection.

\begin{Prop}\label{trop}
Let $X, Y \subset \PP^n$ be irreducible varieties. The tropicalization of the Hadamard product is the Minkowski sum of the tropicalizations:
\[
\textup{trop}(X \hadamardot Y) \, = \, \textup{trop}(X) + \textup{trop}(Y)
\]
as sets.
If also $X \times Y \dashrightarrow X \hadamardot Y$ is generically $\delta$ to $1$, then:
\[
\textup{trop}(X \hadamardot Y) \, = \, \frac{1}{\delta} \, \big {(} \textup{trop}(X) + \textup{trop}(Y) \big{)}
\]
as weighted balanced fans.
\end{Prop}

\begin{proof}
The set theoretic statement is \cite[Proposition 5.5.11]{MS}.  The multiplicities statement follows from \cite[Remark 5.5.2]{MS}.
\end{proof}

\begin{Rmk} \label{explaintrop}
We explain the RHS of Proposition \ref{trop}.  As a set, it is
\[
\textup{trop}(X) + \textup{trop}(Y)  \, = \,  \big { \{ } \textbf{w}_{1} + \textbf{w}_{2} :  \, \textbf{w}_{1} \in \textup{trop}(X), \, \textbf{w}_{2} \in \textup{trop}(Y) \big { \} } \, \subset \, \RR^{n+1} / \RR \textbf{1}.
\]
Now choose refinements of $\text{trop}(X)$ and $\text{trop}(Y)$ so
\[
\big { \{ } \sigma_1 \, + \, \sigma_2 : \, \sigma_1 \text{ is a facet of trop}(X), \, \sigma_2 \text{ is a facet of trop}(Y) \big { \}}
\]
is a polyhedral fan.  Put weights on facets $\sigma$ by
\[
\text{mult}(\sigma) \, = \sum_{\sigma_1, \sigma_2} \text{mult}(\sigma_1) \, \text{mult}(\sigma_2) \, [N_{\sigma} \, : \, N_{\sigma_1} + N_{\sigma_2}].
\]
Here the sum is over all pairs of facets $\sigma_1 , \, \sigma_2$ in $\text{trop}(X), \, \text{trop}(Y)$ respectively such that $\sigma = \sigma_1 + \sigma_2$.  Also $N_{\sigma}$ is the maximal sublattice of $\ZZ^{n+1} / \ZZ \textbf{1}$ parallel to the affine span of $\sigma$, similarly for $N_{\sigma_1}$ and $N_{\sigma_2}$, and $[N_{\sigma} \, : \, N_{\sigma_1} + N_{\sigma_2}]$ is the lattice index.  Now multiply these weights by $\frac{1}{\delta}$.   Finally note the equality of weighted balanced fans is up to common refinement.
\end{Rmk}

\begin{Rmk} \label{multiplefactors}
Proposition \ref{trop} and Remark \ref{explaintrop} generalize to several varieties.  If $X_1, \ldots , X_r \subset \PP^n$ are irreducible and $X_1 \times \ldots \times X_r \dashrightarrow X_1 \hadamardot \ldots \hadamardot X_r$ is generically $\delta$ to $1$, then:
\[
\textup{trop}(X_1 \hadamardot \ldots \hadamardot X_r) \, = \, \frac {1} {\delta} \, \big {(} \textup{trop}(X_1) + \ldots + \textup{trop}(X_r) \big {)}
\]
as weighted balanced fans.  Multiplicities in $\text{trop}(X_1) + \ldots + \text{trop}(X_r)$ are:
\[
\textup{mult}(\sigma) \, = \sum_{\sigma_1, \ldots, \sigma_r} \text{mult}(\sigma_1) \ldots \text{mult}(\sigma_r) \, [N_{\sigma} \, : \, N_{\sigma_1} + \ldots + N_{\sigma_r}].
\]
\end{Rmk}

As a first application of tropical geometry, we can say what dimension we expect Hadamard products to have in general.

\begin{Prop}\label{expecteddim}
Let $X, Y \subset \PP^n$ be irreducible varieties.  Denote by $H \subset (\CC^{*})^{n+1}/\CC^{*}$ the largest subtorus acting on both $X$ and $Y$, and $G \subset (\CC^{*})^{n+1}/\CC^{*}$ the smallest subtorus having a coset containing $X$ and a coset containing $Y$.  Then
\[
\textup{min}\big{\{}\textup{dim}(X) + \textup{dim}(Y) - \textup{dim}(H), \, \textup{dim}(G)\big{\}}
\]
upper bounds $\textup{dim}(X \hadamardot Y)$.  We call it the \textup{expected dimension} of $X \hadamardot Y$.
\end{Prop}

\begin{proof}
Let $V \subset \RR^{n+1} / \RR\textbf{1}$ be the common lineality space of trop($X$) and trop($Y$), i.e., the largest linear subspace contained in all cones of the two fans.  Let $W \subset \RR^{n+1} / \RR\textbf{1}$ be the linear span $\big{\langle}\text{trop}(X), \, \text{trop}(Y)\big{\rangle}$.  We need three facts:

\begin{enumerate}

\item[1.]Tropicalization applied to irreducible varieties preserves dimension.
\item[2.]Proposition \ref{trop}
\item[3.]Here trop($H$) = $V$ and trop($G$) = $W$ (see \cite [\textsection 3] {CTY}).

\end{enumerate}

\vspace{.1cm}

\noindent Now we have:

\vspace{.1cm}

\noindent $ \text{dim}(X \hadamardot Y) \, \, \, \stackrel {1.} {=}  \, \, \, \textup{dim}\big{(}\text{trop}(X \hadamardot Y)\big{)}$

\vspace{.05cm}

\hspace{1.52cm} $ \stackrel {2.} {=}  \, \, \, \textup{dim}\big{(}\textup{trop}(X) + \textup{trop}(Y)\big{)} \hspace{1.5cm}$

\vspace{.05cm}

\hspace{1.52cm} $ = \, \, \, \textup{max} \big{\{}\textup{dim}(C+D) : \, C \textup{ and } D \textup{ are facets in trop}(X) \textup{ and trop}(Y) \big{\}}$

\vspace{.05cm}

\hspace{1.52cm} $\leq \, \, \, \textup{min}\big{\{}\textup{dim}(\textup{trop}(X)) + \textup{dim}(\textup{trop}(Y)) - \textup{dim}(V), \, \textup{dim}(W)\big{\}}$

\vspace{.05cm}

\hspace{1.52cm} $\stackrel {3.} {=}  \, \, \, \textup{min} \big{\{} \textup{dim}(X) + \textup{dim}(Y) - \textup{dim}(H), \, \textup{dim}(G) \big{\}}. $
\end{proof}

It is a tricky matter as to when Hadamard products have deficient dimension.  Here is an example related to {\it tropical A-discriminants} (\cite{DFS}).  We verified it in \texttt{Macaulay2} and Felipe Rinc\'{o}n's \texttt{TropLi} (\cite{R}).

\begin{Ex}\label{dimcounterexample}
Let $X \subset \PP^{11}$  be the Segre product of $\PP^2$ and $\PP^3$, i.e., the variety of $3 \times 4$ matrices with rank $1$.  Let $Y \subset \PP^{11}$ be the linear subpace of $3 \times 4$ matrices with all row and column sums zero.  Then $H = 1$, $G = (\CC^{*})^{12}/\CC^{*}$, so the expected dimension is $\text{min}\{5+5-0, \, 11\} = 10$.\, But $X \hadamardot Y = \{3 \times 4 \text{ matrices with rank} \leq 2 \}$ has dimension $9$.   Here $\text{trop}(X)$ is a $5$-dimensional classical linear space and $\text{trop}(Y)$ is a 5-dimensional tropicalized linear space with $55$ rays and $1656$ maximal cones in its \textit{cyclic Bergman fan} (a refinement of the Gr\"obner subfan).  There is no common lineality space, the linear span is $\RR^{12} / \RR\textbf{1}$, but the Minkowski sum is $9$-dimensional.
\end{Ex}

\vspace{.1cm}

\section{Products of Linear Spaces}
Besides powers of a line, the Hadamard product of linear spaces is not linear.  We start with two numerical examples to illustrate this.
\subsection{Numerical examples}

\begin{Ex}\label{firstex}Let $L$ be the line in $\PP^3$ through points $[2:3:5:7]$ and $[11:13:17:19]$, and $M$ the line through $[23: 29: 31: 37]$ and $[41:43:47:53]$.  Let $x_0,\: x_1,\: x_2,\: x_3$ be the homogenous coordinates of $\PP^3$.  Using ideal elimination in \texttt{Macaulay2}, we compute $L\hadamardot M$ to be the quadric surface defined by:

\[88128x_0^2-89280x_0x_1-5299632x_1^2-817938x_0x_2+8896641x_1x_2-1481805x_2^2\] \[-321510x_0x_3-1777545x_1x_3-54250x_2x_3+116375x_3^2.\]

\vspace{0.15cm}
\noindent Note the ruling on $L\hadamardot M$ is $\{p\hadamardot M: p\in L\} \cup \{q\hadamardot L: q\in M\}$.
\end{Ex}

\begin{Rmk}
In general, for linear spaces $L, M \subset \PP^n$, $L\hadamardot M$ has big Fano schemes, because $L\hadamardot M$ contains the linear spaces $\{p\hadamardot M : p\in L\} \cup \{q\hadamardot L : q\in M\}$.
\end{Rmk}

\begin{Ex}\label{secondex} Let $P$ be the $2$-plane in $\PP^5$ spanned by $[3:1:4:1:5:9], \, [2:6:5:3:5:8]$ and $[9:7:9:3:2:3]$.  In \texttt{Macaulay2}, we compute the Hadamard square $P^{\hadamardot 2}$ to be a cubic hypersuface in $\PP^5$, whose defining polynomial is too big to exhibit here.  We also discover that $P^{\hadamardot 2}$ is singular in codimension $2$, with singular locus $\overline{\{p \hadamardot p : p\in P \}}$. Future work could study singular loci of Hadamard products.  A precedent for secant varieties is \cite[\textsection 5]{MOZ}.
\end{Ex}

\vspace{.1cm}

\subsection{Linear span} One way to quantify nonlinearity of an embedded variety is by the dimension of its linear span.

\begin{Prop}\label{genvander}
Let $L \subset \PP^n$ be a generic linear space of dimension m.  Then the linear span $\langle L^{\hadamardot r} \rangle$ has dimension $\min \{ \binom {m+r}{r} - 1, n \}$.
\end{Prop}

\begin{proof}
Let $L = \langle p_0, p_1, \ldots , p_m \rangle$ and correspondingly write:

\[
L = \textup{rowspan}\begin{pmatrix}
a_{00} & a_{01} & \dots & a_{0n}\\
a_{10} & a_{11} & \dots & a_{1n}\\
\dots & \dots & \dots & \dots \\
a_{m0} & a_{m1} & \dots & a_{mn}
\end{pmatrix}.
\]

\noindent By the analogue of Lemma \ref{linenoclosurelem} for higher dimensional $L$, we have:
\[
\langle L^{\hadamardot r} \rangle = \langle p_{0}^{\hadamardot r_0} \hadamardot p_{1}^{\hadamardot r_1} \hadamardot \ldots \hadamardot p_{m}^{\hadamardot r_m}  :  r_i \in \mathbb{Z}_{\geq 0} \text{  and  } r_0 + r_1 + \ldots + r_m = r \rangle,
\]
and correspondingly:

\[
L^{\hadamardot r} = \textup{rowspan}\begin{pmatrix}
a_{00}^{r} & a_{01}^{r} & \dots & a_{0n}^{r}\\
\dots & \dots & \dots & \dots \\
\prod a_{i0}^{r_i} &  \prod a_{i1}^{r_i} & \dots & \prod a_{im}^{r_i} \\
\dots & \dots & \dots & \dots \\
a_{m0}^{r} & a_{m1}^{r} & \dots & a_{mn}^{r}\\
\end{pmatrix}.
\]
The maximal minors of this $\binom {m+r}{r} \times (n+1)$ matrix are the \textit{generalized Vandermonde determinants} from \cite{AA}.  To finish the proof, we note they are non-zero polynomials in $\mathbb{C}[a_{ij} : 0 \leq i \leq m \, , \, 0 \leq j \leq n]$, because each of their coefficients is $\pm$1.  Indeed, they have no like terms as sums over the symmetric group.
\end{proof}

\begin{Cor} \label{spandim}
Let $L \subset \PP^n$ be a generic linear space of dimension $m$.  Assume $n \geq \binom{m+r}{r} - 1$.  Consider points $p_i, q_j \in L$ for $1 \leq i, j \leq r$ such that $p_1 \hadamardot \ldots \hadamardot p_r =  q_1 \hadamardot \ldots \hadamardot q_r$.  Then $\{p_1, \ldots , p_r \} = \{q_1, \ldots , q_r \}$.
\end{Cor}

\begin{proof}
Choose a basis $L = \langle t_0, \ldots , t_m \rangle$. Choose lifts $\tilde{p_i}, \tilde{q_i}, \tilde{t_i}$ of $p_i, q_i, t_i$ to the affine cone $\tilde{L} \subset \mathbb{A}^{n+1}$ over $L$ so that $\tilde{p_1} \hadamardot \ldots \hadamardot \tilde{p_r} =  \tilde{q_1} \hadamardot \ldots \hadamardot \tilde{q_r}$.  Expand:

\vspace{.15cm}
\begin{itemize}
\item $\tilde{p_i} = \sum_{j=0}^{m} \alpha_{ij}\tilde{t_j}$ \vspace{.1cm}
\item $\tilde{q_i} = \sum_{j=0}^{m} \beta_{ij}\tilde{t_j}.$
\end{itemize}
\vspace{.15cm}

\noindent Define a ring homomorphism
\begin{center}
$\Phi : \,  \CC[T_0, \ldots, T_m] \, \longrightarrow \, \CC^{\times(n+1)}$

\hspace{1.25cm} $T_i \longmapsto t_i$.
\end{center}
Here the domain is the polynomial ring, and the codomain is the direct product of $n+1$ copies of $\CC$.  Then:

\begin{center}
\vspace{.15cm}
$ \Phi \, \big{(} \, (\sum_{j=0}^{m} \alpha_{1j}T_j) \, \ldots \, (\sum_{j=0}^{m} \alpha_{rj}T_j)  \, \big{ )} \, \, \,  = \, \, \, \Phi \, \big{(} \, (\sum_{j=0}^{m} \beta_{1j}T_j) \, \ldots \, (\sum_{j=0}^{m} \beta_{rj}T_j) \, \big{ )} $
\vspace{.15cm}
\end{center}

\noindent By Proposition \ref{genvander}, $\text{ker}(\Phi)$ contains no homogenous polymomials of degree $r$.  Hence:

\begin{center}
\vspace{.15cm}
$ (\sum_{j=0}^{m} \alpha_{1j}T_j) \, \ldots \, (\sum_{j=0}^{m} \alpha_{rj}T_j) \, \, \, = \, \, \,   (\sum_{j=0}^{m} \beta_{1j}T_j) \, \ldots \, (\sum_{j=0}^{m} \beta_{rj}T_j). $
\vspace{.15cm}
\end{center}

\noindent Since $\CC[T_0, \ldots, T_m]$ is a unique factorization domain, the corollary follows.
\end{proof}

\begin{Rmk}\label{Landsberg}
Corollaries \ref{final} and \ref{spandim} correspond to \textit{identifiablity} (\cite{BC12, CC2, Str}) in the theory of secant varieties.  A direction for future research is to look for $(\text{dim}(L), r, n)$ such that $L^{\hadamardot r} \subset \PP^n$ is generically \textit{not} identifiable.  For secant varieties of symmetric tensors, a classification is \cite[Theorem 3.3.3.1]{JM}.
\end{Rmk}

\begin{Rmk}\label{identify}
Proposition \ref{genvander} and Corollary \ref{spandim} generalize to several linear spaces.  Let $L_1, \ldots, L_k \subset \PP^n$ be generic linear spaces of dimensions $m_1, \ldots, m_k$.  The linear span $\langle L_1^{r_1} \hadamardot \ldots \hadamardot L_k^{r_k} \rangle$ has dimension $\text{min} \big{\{} \binom {m_1 + r_1}{r_1} \ldots \binom {m_k + r_k}{r_k} - 1, \, \, n \big {\}}$.  If $n \geq \binom {m_1 + r_1}{r_1} \ldots \binom {m_k + r_k}{r_k} - 1$, then
$L_1^{r_1} \hadamardot \ldots \hadamardot L_k^{r_k}$ is identifiable.  The proofs are similar.
\end{Rmk}

\vspace{.1cm}

\subsection{Degree formula}

As a second application of tropical geometry, we now prove a general degree formula for Hadamard products of linear spaces, when the ambient projective space is of sufficiently high dimension.  This is one of our main results.

\begin{Thm}\label{degformula}
Let $L_1, L_2, \ldots, L_r \subset \mathbb{P}^n$ be generic linear spaces of dimensions $m_1, m_2, \ldots, m_r$, respectively.   Set $m = m_1 + m_2 + ... + m_r$ and $d = {m_1 + m_2 + \ldots + m_r \choose m_1, \, m_2, \, \ldots \, , \, m_r}$.  Assume $n \gg 0$.  Then $L_1 \hadamardot L_2 \hadamardot \ldots \hadamardot L_r$ has dimension $m$ and degree
\[
\begin{cases}
   d & \text{if } L_i \text{ are pairwise distinct} \\
   \frac{d}{r!}  & \text{if } L_i \text{ are the same.}
\end{cases}
\]
In full generality, when $L_i$ form a multiset with multiplicites $r_1, \ldots, r_k$ (where $r_1 + \ldots + r_k = r$), the dimension is $m$ and the degree is $ \frac{d} {(r_1!) \,\ldots \, (r_k!)}$.
\end{Thm}

\begin{proof}[Proof (in full generality).]  Generically $L_i$ has Pl\"ucker coordinates all nonzero, so $\text{trop}(L_i)$ equals the \textit{standard tropical linear space} $\Lambda_{m_i}$ of dimension $m_i$ (\cite[Example 4.2.13]{MS}):
\[
\textup{trop}(L_i) \, = \, \Lambda_{m_i} \, \, = \bigcup_{0 \leq j_1 < \ldots < j_{m_i} \leq n} \textup{pos}(\textbf{e}_{j_1}, \ldots , \textbf{e}_{j_{m_i}}).
\]
Here $\textbf{e}_0, \ldots , \textbf{e}_n$ are the images in $\RR^{n+1} / \RR\textbf{1}$ of the standard basis vectors, pos denotes positive span, and all multiplicites are $1$.

By Remark \ref{multiplefactors}, $\text{trop}(L_1) \, + \, \ldots \, + \, \text{trop}(L_r)$ has support the same as $\Lambda_m$, and for each facet:
\[
\textup{mult}(\sigma) \, \, = \sum_{\sigma_1, \ldots, \sigma_r} \text{mult}(\sigma_1) \ldots \text{mult}(\sigma_r) \, [N_{\sigma} \, : \, N_{\sigma_1} + \ldots + N_{\sigma_r}] \, \, = \sum_{\sigma_1, \ldots, \sigma_r} 1 \, \, = \, \, d.
\]

By Remark \ref{identify}, $L_1 \, \times \, \ldots \, \times \, L_r \dashrightarrow L_1 \, \hadamardot \, \ldots \, \hadamardot \, L_r$ is generically $(r_1!) \,\ldots \, (r_k!)$ to $1$.  So by Remark \ref{multiplefactors}:
\[
\textup{trop}(L_1 \hadamardot \ldots \hadamardot L_r) = \frac{\tiny{d}} {(\tiny {r_1!) \,\ldots \, (r_k!)}} \, \, \,  \Lambda_m.
\]

It follows that $L_1 \hadamardot \ldots \hadamardot L_r$ has dimension $m$ and degree (\cite[Corollary 3.6.16]{MS}):
\begin{align*}
\textup{deg}(L_1 \hadamardot \ldots \hadamardot L_r) \, & = \, \, \textup{mult}_{\textbf{0}} \, \big {(}\, \textup{trop}(L_1 \hadamardot \ldots \hadamardot L_r) \, \,  \cap{}_{\textup{st}} \, \, \Lambda_{n-m}\big {)} \\
& = \, \, \textup{mult}_{\textbf{0}} \, \big {(}\, \frac{\tiny{d}} {(\tiny {r_1!) \,\ldots \, (r_k!)}} \, \, \,  \Lambda_m \, \,  \cap{}_{\textup{st}} \, \, \Lambda_{n-m}\big {)}  \\
& = \, \, \frac{\tiny{d}} {(\tiny {r_1!) \,\ldots \, (r_k!)}} \, .
\end{align*}
\noindent Here we take \textit{stable intersection} (\cite{JY}) with the standard tropical linear space of complementary dimension, and then measure the multiplicity of the origin.
\end{proof}

\begin{Rmk}\label{suffbig}
We can be explicit about $n \gg 0$ in Theorem \ref{degformula}.  From the application of Remark \ref{identify}, the hypothesis is $n \geq \binom {m_1 + r_1}{r_1} \ldots \binom {m_k + r_k}{r_k} - 1$.
\end{Rmk}

We can extend Theorem \ref{degformula} to include the case of negative exponents.  Given a linear space $L \subset \PP^n$, its \textit{reciprocal linear space} is:

$$L^{-1} = \overline{\{[a_{0}^{-1}: \ldots : a_{n}^{-1}] \in \PP^n: [a_{0}: \ldots: a_{n}] \in L \setminus \Delta_{n-1} \}}.$$

\noindent These varieties have good combinatorial properties.  In \cite{HuhK}, they were used to prove log-concavity of the coefficients of the characteristic polynomial of every realizable matroid.  From \cite{PS}, the defining ideal of $L^{-1}$ has all initial ideals squarefree, and a universal Gr\"obner basis in bijection with the set of circuits in the matroid realized by $L$.  For familiar examples, note rational normal curves are reciprocal lines.

%

\begin{Cor}\label{reciprocal}
Use the notation of Theorem \ref{degformula}.  Additionally, let $\tilde{L}_1, \ldots, \tilde{L}_{s}$ be generic linear spaces of dimensions $\tilde{m}_1, \ldots, \tilde{m}_{s}$, and set $\tilde{m} = \tilde{m}_1 + \ldots + \tilde{m}_s$ and $\tilde{d} = {\tilde{m}_1 + \ldots + \tilde{m}_s \choose \tilde{m}_1, \, \ldots \, , \, \tilde{m}_s}$.  Assume $n \gg 0$ and $\tilde{L}_{i}$ form a multiset with multiplicities $s_1, \ldots, s_{\ell}$.  Then $L_1 \hadamardot \ldots \hadamardot L_r \hadamardot \tilde{L}^{-1}_1 \hadamardot \ldots \hadamardot \tilde{L}^{-1}_s$ has dimension $m+\tilde{m}$ and degree:
\[
\binom{n-m}{\tilde{m}} \cdot \frac{d}{(r_1!) \,\ldots \, (r_k!)} \cdot \frac{\tilde{d}}{(s_1!) \,\ldots \, (s_\ell!)}.
\]

\end{Cor}

\begin{proof}
The tropicalization of the reciprocal linear space $\tilde{L}_{i}^{-1}$ equals:
\[
\text{trop}(\tilde{L}_{i}^{-1}) = -\Lambda_{\tilde{m}_i} = \bigcup_{0 \leq j_1 < \ldots < j_{\tilde{m}_i} \leq n} \textup{neg}(\textbf{e}_{j_1}, \ldots , \textbf{e}_{j_{\tilde{m}_i}}).
\]
Here neg denotes negative span, and all multiplicites are $1$.  It follows that in
\[
\text{trop}(L_1) + \ldots \text{trop}(L_r) + \text{trop}(\tilde{L}_{1}^{-1}) + \ldots + \text{trop}(\tilde{L}_{s}^{-1}),
\]
the cones are:
\begin{equation}\label{plusminuscones}
\text{pos}(\textbf{e}_i: i \in I) + \text{neg}(\textbf{e}_j: j \in J)  \tag{*}
\end{equation}
for $I, J \subset \{0, \ldots ,n \}$ with $|I|=m, \, |J|=\tilde{m}$ and $I \cap J = \emptyset$.  Similar to in the proof of Theorem \ref{degformula}, the multiplicities of these cones are all $d\tilde{d}$.
Using Remark \ref{identify}, we can clear denominators to deduce that:
\[
L_1 \times \ldots \times L_r \times \tilde{L}_{1}^{-1} \times \ldots \times \tilde{L}_{s}^{-1} \dashrightarrow L_1 \hadamardot \ldots \hadamardot L_r \hadamardot \tilde{L}^{-1}_1 \hadamardot \ldots \hadamardot \tilde{L}^{-1}_s
\]
is generically $(r_1!) \,\ldots \, (r_k!)(s_1!) \,\ldots \, (s_\ell!)$ to 1.  So by Remark \ref{multiplefactors}:
\[
\text{trop}(L_1 \hadamardot \ldots \hadamardot L_r \hadamardot \tilde{L}^{-1}_1 \hadamardot \ldots \hadamardot \tilde{L}^{-1}_s)
\]
consists of the cones \eqref{plusminuscones}, each with multiplicity $\frac{d}{(r_1!) \,\ldots \, (r_k!)} \cdot \frac{\tilde{d}}{(s_1!) \,\ldots \, (s_\ell!)}$.  It follows the dimension of the Hadamard product and its tropicalization are $m+\tilde{m}$.

For the degree, take stable intersection with $\Lambda_{n-m-\tilde{m}}$.  To compute the multiplicity of $\textbf{0}$ using formula ([2]) of \cite{JY}, pick $v \in \RR^{n+1}/\RR\textbf{1}$ to be generic inside the octant $\text{pos}(\textbf{e}_1, \ldots, \textbf{e}_m) + \text{neg}(\textbf{e}_{m+1}, \ldots, \textbf{e}_{n})$.  Then there are $\binom{n-m}{\tilde{m}}$ summands in that formula: indeed for each $I \subset \{m+1, \ldots, n\}$, the pair $\sigma_1=\text{pos}(\textbf{e}_1, \ldots, \textbf{e}_m) + \text{neg}(\textbf{e}_{i}: i \in I)$ and $\sigma_2 = \text{neg}(\textbf{e}_{i}: i \in \{m+1, \ldots, n\} \setminus I)$ contributes a term, namely:
\[
\frac{d}{(r_1!) \,\ldots \, (r_k!)} \cdot \frac{\tilde{d}}{(s_1!) \,\ldots \, (s_\ell!)} \cdot 1 \cdot 1.
\]
This gives a total of $\binom{n-m}{\tilde{m}} \cdot \frac{d}{(r_1!) \,\ldots \, (r_k!)} \cdot \frac{\tilde{d}}{(s_1!) \,\ldots \, (s_\ell!)}$, and the corollary follows.
\end{proof}

We conclude with two Pl\"ucker formulas for products in the regime of Theorem \ref{degformula}.  They consist of bracket polynomials, like in Proposition \ref{Plucker}, and correspond to numerical Examples \ref{firstex} and \ref{secondex}, respectively.

\begin{Ex}\label{distinctlines}
Let $L$ and $M$ be generic distinct lines in $\PP^3$, with Pl\"ucker coordinates $ \textup{[}ij\textup{]} := \textup{[}i, j\textup{]}_{L}$ and $\{ij\} := \textup{[}i, j\textup{]}_{M}$ respectively.  Then $L\hadamardot M$ is the quadric surface in $\PP^{3}$ cut out by:

\vspace{0.25cm}
\begin{center}
{\footnotesize
\Lonetwo \Lonethree \Ltwothree \Monetwo \Monethree \Mtwothree \, $x_{0}^2$  $+$ \Lzerotwo \Lzerothree \Ltwothree \Mzerotwo \Mzerothree \Mtwothree \, $x_{1}^2$

\vspace{0.08cm}

$+$ \Lzeroone \Lzerothree \Lonethree \Mzeroone \Mzerothree \Monethree \, $x_{2}^2$ $+$ \Lzeroone \Lzerotwo \Lonetwo \Mzeroone \Mzerotwo \Monetwo \, $x_{3}^2$

\vspace{0.08cm}

$-$ \Ltwothree \Mtwothree $\big ($\Lzerotwo \Lonethree \Mzerothree \Monetwo \hspace{.07em} $+$  \Lzerothree \Lonetwo \Mzerotwo \Monethree $\big ) \, x_0x_1 $

\vspace{0.08cm}

$+$ \Lonethree \Monethree $\big ($\Lzeroone \Ltwothree \Mzerothree \Monetwo  \hspace{.07em} $+$  \Lzerothree \Lonetwo \Mzeroone \Mtwothree $\big ) \, x_0x_2 $

\vspace{0.08cm}

$-$ \Lonetwo \Monetwo $\big ($\Lzeroone \Ltwothree \Mzerotwo \Monethree \hspace{.07em} $+$  \Lzerotwo \Lonethree \Mzeroone \Mtwothree $\big) \, x_0x_3 $

\vspace{0.08cm}

$-$ \Lzerothree \Mzerothree $\big ($\Lzeroone \Ltwothree \Mzerotwo \Monethree \hspace{.07em} $+$  \Lzerotwo \Lonethree \Mzeroone \Mtwothree $\big ) \, x_1x_2 $

\vspace{0.08cm}

$+$ \Lzerotwo \Mzerotwo $\big($\Lzeroone \Ltwothree \Mzerothree \Monetwo \hspace{.07em} $+$  \Lzerothree \Lonetwo \Mzeroone \Mtwothree $\big ) \, x_1x_3 $

\vspace{0.08cm}

$-$ \Lzeroone \Mzeroone $\big($\Lzerotwo \Lonethree \Mzerothree \Monetwo \hspace{.07em} $+$  \Lzerothree \Lonetwo \Mzerotwo \Monethree $\big ) \, x_2x_3 $}.

\end{center}


\vspace{.1in}

To prove this formula, substitute $\textup{[}ij\textup{]} = a_{0i}a_{1j} - a_{0j}a_{1i}$, $\{ij\} = b_{0i}b_{1j} - b_{0j}b_{1i}$, and $x_{i} = (\lambda_{0}a_{0i} + \lambda_{1}a_{1i})(\mu_{0}b_{0i} + \mu_{1}b_{1i})$ in \texttt{Macaulay2} , and get identically zero.

To find this formula, note the incidence variety
\[
\mathds{X} = \overline { {\big \{} (p, L, M) \in \PP^{3} \times \text{Gr}(2, 4) \times \text{Gr}(2, 4) : \, \,  p \in L \hadamardot M {\big \}}}
\]
has three group actions:
\begin{itemize}
\item $\mathfrak{S}_{2}$ acts by switching $L$ and $M$
\item $\mathfrak{S}_{4}$ acts by permuting the homogeneous coordinates of $\PP^{3}$
\item $(\CC^{*})^{4} / \CC^{*}$ acts by scaling the homogeneous coordinates of $\PP^{3}$.
\end{itemize}
So, the defining equation of $\mathds{X}$ is $(\mathfrak{S}_{2} \times \mathfrak{S}_{4})$-symmetric and $\mathbb{Z}^3$-multihomogeneous.  Also, specializing $L = M$ should give the square of the linear equation in Corollary \ref{lineareqn}.  Lastly, guess the formula has integer coefficients, on the basis of Example \ref{firstex}.  Putting these clues together, it is straightforward to find the formula.
\end{Ex}

%
%
%
%

\begin{Ex}\label{twoplanes}
Let $P$ be a generic 2-plane in $\PP^5$, with Pl\"ucker coordinates $ \textup{[}ijk\textup{]} := \textup{[}i, j, k\textup{]}_{P}$.  Then $P^{\hadamardot 2}$ is the cubic hypersurface in $\PP^5$.  In a defining equation, the coefficient of $x_0^3$ is {\small $-(-1)^{0+0+0}$} times:

\vspace{0.25cm}

{\centering
{\footnotesize

\noindent \hspace{4cm} $
\textup{[}123\textup{]}
\textup{[}124\textup{]}
\textup{[}125\textup{]}
\textup{[}134\textup{]}
\textup{[}135\textup{]}
\textup{[}145\textup{]}
\textup{[}234\textup{]}
\textup{[}235\textup{]}
\textup{[}245\textup{]}
\textup{[}345\textup{]}
$}.}

\vspace{0.25cm}

\noindent The coefficient of $x_0^2x_1$ is {\small $-(-1)^{0+0+1}$} times:

\vspace{0.25cm}

{\tiny
\noindent $ {\Big (}\textup{[}023\textup{]}
\textup{[}045\textup{]}
\textup{[}124\textup{]}
\textup{[}125\textup{]}
\textup{[}134\textup{]}
\textup{[}135\textup{]} \, \, + \, \,
\textup{[}024\textup{]}
\textup{[}035\textup{]}
\textup{[}123\textup{]}
\textup{[}125\textup{]}
\textup{[}134\textup{]}
\textup{[}145\textup{]} \, \, + \, \,
\textup{[}025\textup{]}
\textup{[}034\textup{]}
\textup{[}123\textup{]}
\textup{[}124\textup{]}
\textup{[}135\textup{]}
\textup{[}145\textup{]} {\Big)}
\textup{[}234\textup{]}
\textup{[}235\textup{]}
\textup{[}245\textup{]}
\textup{[}345\textup{]}
$}.

\vspace{0.25cm}

\noindent The coefficient of $x_0x_1x_2$ is {\small $-(-1)^{0+1+2}$} times:

\vspace{0.25cm}

\begin{center}
{\tiny
$
\textup{[}013\textup{]}
\textup{[}024\textup{]}
\textup{[}134\textup{]}
\textup{[}234\textup{]}
\textup{[}125\textup{]}
\textup{[}035\textup{]}
\textup{[}235\textup{]}
\textup{[}045\textup{]}
\textup{[}145\textup{]}
\textup{[}345\textup{]} \, \, + \, \,
\textup{[}013\textup{]}
\textup{[}124\textup{]}
\textup{[}034\textup{]}
\textup{[}234\textup{]}
\textup{[}025\textup{]}
\textup{[}135\textup{]}
\textup{[}235\textup{]}
\textup{[}045\textup{]}
\textup{[}145\textup{]}
\textup{[}345\textup{]} \, \, + $

\vspace{0.16cm}

$
\textup{[}023\textup{]}
\textup{[}014\textup{]}
\textup{[}134\textup{]}
\textup{[}234\textup{]}
\textup{[}125\textup{]}
\textup{[}035\textup{]}
\textup{[}135\textup{]}
\textup{[}045\textup{]}
\textup{[}245\textup{]}
\textup{[}345\textup{]} \, \, + \, \,
\textup{[}023\textup{]}
\textup{[}124\textup{]}
\textup{[}034\textup{]}
\textup{[}134\textup{]}
\textup{[}015\textup{]}
\textup{[}135\textup{]}
\textup{[}235\textup{]}
\textup{[}045\textup{]}
\textup{[}245\textup{]}
\textup{[}345\textup{]} \, \, + $

$
\textup{[}123\textup{]}
\textup{[}014\textup{]}
\textup{[}034\textup{]}
\textup{[}234\textup{]}
\textup{[}025\textup{]}
\textup{[}035\textup{]}
\textup{[}135\textup{]}
\textup{[}145\textup{]}
\textup{[}245\textup{]}
\textup{[}345\textup{]} $}.
\end{center}

\vspace{0.25cm}

\noindent To get the other coefficients, act on the indices by $\mathfrak{S}_{6}$.  We find and prove this as in Example \ref{distinctlines}, here specializing $P = L^{\hadamardot 2}$ where $L$ is a line.
\end{Ex}

\vspace{.25cm}

\noindent \textbf{Acknowledgements.} The authors are grateful to Bernd Sturmfels for suggesting that we investigate multiplicative geometry and reading an earlier draft.  The third author also thanks Justin Chen and Steven Sam for helpful conversations.


\begin{thebibliography}{99}

\bibitem{AA} L.F.T. Alonso and C. D'Andrea, {\it Tropicalization and irreducibility of generalized Vandermonde determinants}, Proceedings of the American Mathematical Society {\bf 137} (2009),  3647--3656.

\bibitem{BC12} C. Bocci and L. Chiantini, {\it On the identifiability of binary Segre products}, J. Alg. Geom. {\bf 22} (2013) 1--11.
    
\bibitem{BH10} C. Bocci and B. Harbourne, {\it Comparing powers and symbolic powers of ideals}, J. Alg. Geom. {\bf 19} (2010) 399--417.

\bibitem{CGVT} E. Carlini, E. Guardo and A. Van Tuyl, {\it Star configurations on generic hypersurfaces}, J. Algebra {\bf 407} (2014), 1--20.

\bibitem{CVT} E. Carlini and A. Van Tuyl, {\it Star configuration points and generic plane curves}, Proc. Amer. Math. Soc. {\bf 139} (2011), no. 12, 4181--4192.

\bibitem{CC2} L. Chiantini and C. Ciliberto,  {\it On the concept of $k$--secant order of a variety}, J. London Math. Soc. {\bf 73} (2006) 436--454.

\bibitem{CLO} D.A. Cox, J. Little and D. O'Shea, {\it Using Algebraic Geometry}, second ed., Springer-Verlag (2005).

\bibitem{CMS}  M.A. Cueto, J. Morton and B. Sturmfels, {\it Geometry of the restricted Boltzmann machine}, Algebraic Methods in Statistics and Probability (eds. M. Viana and H. Wynn), American Mathematicals Society, Contemporary Mathematics {\bf 516} (2010) 135--153.

\bibitem{CTY} M.A. Cueto, E.A. Tobis and J. Yu, {\it An implicitization challenge for binary factor analysis}, J. Symbolic Comput. \textbf{45} (2010), no. 12, 1296--1315.

\bibitem{DFS} A. Dickenstein, E.M. Feichtner and B. Sturmfels, {\it Tropical discriminants}, J. Amer. Math. Soc. \textbf{20} (2007) 1111--1133.

\bibitem{Eis} D. Eisenbud, {\it Commutative Algebra with a View Toward Algebraic Geometry}, Springer (1999).

\bibitem{GHM} A.V. Geramita, B. Harbourne and J. Migliore, {\it Star configurations in $\PP^n$}, J. Algebra \textbf{376} (2013), 279--299.

\bibitem{M2} D.R. Grayson and M.E. Stillman: {\it Macaulay2, a software system for research in algebraic geometry}.  Available at \url{http://www.math.uiuc.edu/Macaulay2/}.

\bibitem{Hart} R. Hartshorne, {\it Algebraic Geometry}, Graduate Texts in Mathematics vol. 52, Springer-Verlag (1977).

\bibitem{H1} R.A. Horn, {\it Topics in Matrix Analysis}, Cambridge University Press, 1994.

\bibitem{HM} R.A. Horn and R. Mathias, {\it Block-matrix generalizations of Schur's basic theorems on Hadamard products}, Linear Algebra Appl. \textbf{172} (1992) 337--346.

\bibitem{HuhK} J. Huh and E. Katz, {\it Log-concavity of characteristic polynomials and the Bergman fan of matroids}, Mathematische Annalen \textbf{354} (2012), 1103--1116.

\bibitem{IS} N.O. Ilten and H. S\"uss, {\it Polarized complexity-1 T-varieties}, Michigan Math. J. \textbf{60} (2011), no. 3, 561--578.

\bibitem{JY} A. Jensen and J. Yu, {\it Stable intersections of tropical varieties}, arXiv:1309.7064v3.

\bibitem{JM} J.M. Landsberg, {\it Tensors: geometry and applications}, Graduate Studies in Mathematics, vol. 128, AMS, Providence, RI (2012).

\bibitem{Liu} S. Liu, {\it Inequalities involving Hadamard products of positive semidefinite matrices}, J. Math. Ana. Appl. \textbf{243} (2000b) 458--463.

\bibitem{Liu2} S. Liu, {\it On the Hadamard Product of Square Roots of Correlation Matrices}, Econometric Theory \textbf{18/19} (2002b/2003) 1007/703--704.

\bibitem{LNP}  S. Liu and H. Neudecker, W. Polasek, {\it  The Hadamard product and some of its applications in statistics}, Statistics \textbf{26} (1995a) no. 4, 365--373.

\bibitem{LNP2}  S. Liu and H. Neudecker, W. Polasek, {\it The heteroskedastic linear regression model and the Hadamard product -- a note}, J. Econometrics \textbf{68} (1995b) 361--366.

\bibitem{LT} S. Liu and G. Trenkler, {\it Hadamard, Khatri-Rao, Kronecker and other matrix products},  Int. Jour. of Information and Systems Sciences {\bf 4} (1), (2008), 160--177.

\bibitem{MS} D. Maclagan and B. Sturmfels, {\it Introduction to tropical geometry}, Graduate Studies in Mathematics vol. 161, AMS, Providence, RI (2015).

\bibitem{MOZ} M. Michalek, L. Oeding and P. Zwiernik, {\it Secant cumulants and toric geometry}, International Mathematics Research Notices, rnu056 (2014).

\bibitem{MillS} E. Miller and B. Sturmfels, {\it Combinatorial Commutative Algebra}, Graduate Texts in Mathematics vol. 227, Springer (2005).

\bibitem{M} D. Mumford, {\it The Red Book of Varieties and Schemes}, Springer-Verlag (1999).

\bibitem{PS} N. Proudfoot and D. Speyer, {\it A broken circuit ring}, Beitr\"age Algebra Geom. \textbf{47} (2006), no. 1, 161--166.

\bibitem{R} F. Rinc\'{o}n, {\it Computing Tropical Linear Spaces}, J. Symbolic Comput. \textbf{51} (2013), 86--98.

\bibitem{Str} V. Strassen, {\it Rank and optimal computation of generic tensors},  Linear Algebra Appl. \textbf{52} (1983) 645--685.

\bibitem{S} B. Sturmfels, {\it Equations defining toric varieties}, Algebraic Geometry - Santa Cruz 1995, Proc. Sympos. Pure Math. \textbf{62}, Part 2, AMS, Providence, RI, (1997), 437--449.

\end{thebibliography}
\end{document}